\documentclass[preprint,12pt]{elsarticle}
\usepackage[T1]{fontenc}
\usepackage[utf8]{inputenc}
\usepackage{lmodern}
\usepackage{amsmath,amssymb,amsthm,mathtools,mathrsfs,bm}
\usepackage{enumitem}
\usepackage{microtype}
\usepackage{hyperref}
\usepackage[nameinlink,capitalize]{cleveref}
\hypersetup{colorlinks=true,linkcolor=blue,citecolor=blue,urlcolor=blue}

\newtheorem{theorem}{Theorem}[section]
\newtheorem{proposition}[theorem]{Proposition}
\newtheorem{corollary}[theorem]{Corollary}
\newtheorem{lemma}[theorem]{Lemma}
\theoremstyle{definition}
\newtheorem{definition}[theorem]{Definition}
\newtheorem{remark}[theorem]{Remark}
\newtheorem{example}[theorem]{Example}

\newcommand{\R}{\mathbb R}

\newcommand{\hookc}{\mathrel{\hookrightarrow\!\hookrightarrow}}

\journal{Journal of Functional Analysis}

\begin{document}
\begin{frontmatter}
\title{Compact Embeddings of Vector-Valued Morrey Spaces}
\author[cp]{Rishad Shahmurov}
\ead{rshahmurov@crimson.ua.edu}
\author[abu]{Veli Shahmurov\corref{cor1}}
\ead{veli.sahmurov@antalya.edu.tr}
\cortext[cor1]{Corresponding author}
\address[cp]{Cellular Products Research and Development, 595 East Crossville Road, Roswell, GA 30075, USA}
\address[abu]{Antalya Bilim University, Antalya, Turkey}

\begin{abstract}
We develop compactness and defect-of-compactness results for Banach-valued evolution classes with Morrey control in time.  For
\[
\begin{aligned}
\mathbb W_M^{p,\lambda}(0,T;E_0,E_1)
=\{u\in\mathcal M^{p,\lambda}(0,T;E_0):\;&
 u'\in\mathcal M^{p,\lambda}(0,T;E_1)\},\\[-1mm]
&0<\lambda<1.
\end{aligned}
\]
the exact trace exponent
\[
\theta=1-\frac{1-\lambda}{p}
\]
governs both continuity and compactness.  If $E_0\hookrightarrow\!\hookrightarrow E\hookrightarrow E_1$, bounded sets are compact in the \emph{same} Morrey space $\mathcal M^{p,\lambda}(0,T;E)$, not merely in $L^p(0,T;E)$.  In contrast, compactness in $C([0,T];E)$ holds if and only if the exact trace space $(E_1,E_0)_{\theta,\infty}$ embeds compactly into $E$.  We also obtain compact lower-order H\"older embeddings and a sharp Hilbert-triple threshold.

On unbounded domains we prove a tightness criterion for global Morrey compactness.  In the Hilbert-valued case we establish, by a direct time-averaging argument, cocompactness modulo spatial translations and a translation profile decomposition whose remainder vanishes in every strictly subcritical time-Morrey--Sobolev target; a uniform little-Morrey condition removes the loss in the time exponent.  At the doubly critical endpoint, heat and whole-space Stokes dynamics force balanced parabolic profiles, and the remainder vanishes in $\mathcal M_t^{2,\lambda}L_x^{2^*}$.  Finally, in three-dimensional Navier--Stokes we show that the classical nonlinear profile decomposition has a scale-sensitive Morrey refinement: the remainder is small in $\mathcal M_t^{p,1-p/4}L_x^6$ for every $2\le p<4$, and orthogonal profiles have vanishing interaction in the corresponding Morrey forcing space.
\end{abstract}

\end{frontmatter}

\section{Introduction}

Compactness is one of the main mechanisms by which a priori estimates for evolution equations become existence theorems.  The classical Aubin--Lions lemma and Simon's refinement combine a compact spatial embedding with time regularity to obtain strong compactness in Bochner spaces; see Aubin~\cite{Aubin1963} and Simon~\cite{Simon1987}.  Rossi--Savar\'e~\cite{RossiSavare2003} subsequently placed this circle of ideas in a tightness and integral-equicontinuity framework.  These theories are naturally formulated in $L^p(0,T;E)$.

The aim of this paper is to understand what changes when the time variable is controlled in a Morrey norm.  Such a bound retains scale information on every time interval rather than only a global integral.  We show that this additional information survives compactness and, at the same time, reveals a sharp distinction between compactness in a Morrey topology and compactness at each time.

Let
\[
\mathbb W_M^{p,\lambda}(0,T;E_0,E_1)
=\left\{u\in\mathcal M^{p,\lambda}(0,T;E_0):
 u'\in\mathcal M^{p,\lambda}(0,T;E_1)\right\},
\]
where $E_0\hookrightarrow E_1$, $1<p<\infty$, and $0<\lambda<1$.  Put
\[
\theta=1-\frac{1-\lambda}{p},
\qquad
X_\theta=(E_1,E_0)_{\theta,\infty}.
\]
The trace theorem~\cite{ShahmurovTrace2026} identifies $X_\theta$ as the exact trace space and supplies a bounded extension of arbitrary trace data.

Our first main result is a Morrey-strengthened Aubin--Lions theorem.  If
\[
E_0\hookrightarrow\!\hookrightarrow E\hookrightarrow E_1,
\]
then
\[
\mathbb W_M^{p,\lambda}(0,T;E_0,E_1)
\hookrightarrow\!\hookrightarrow
\mathcal M^{p,\lambda}(0,T;E).
\]
Thus strong convergence takes place in the \emph{same} scale-sensitive Morrey norm.  The compactness literature around Aubin--Lions is predominantly formulated in Bochner, fractional-time, weighted, or evolving-space topologies; the result below is instead a same-time-Morrey conclusion and is proved by upgrading the classical $L^p$ compactness through the Morrey time modulus.

Pointwise-in-time compactness has a different exact criterion.  Whenever $E\hookrightarrow E_1$,
\[
\boxed{
\mathbb W_M^{p,\lambda}(0,T;E_0,E_1)
\hookrightarrow\!\hookrightarrow C([0,T];E)
\quad\Longleftrightarrow\quad
X_\theta\hookrightarrow\!\hookrightarrow E.
}
\]
The sufficiency follows from the trace estimate, Ehrling's lemma, and Arzel\`a--Ascoli.  The necessity uses the trace extension: every noncompact bounded sequence in $X_\theta$ lifts to a bounded sequence of Morrey evolutions with the same defect at $t=0$.  Hence the exact trace space is not merely a regularity space; it is the complete obstruction to $C_tE$ compactness.

Below the trace level we obtain the mixed scale
\[
\mathbb W_M^{p,\lambda}
\hookrightarrow
C^{0,\theta-\eta}\bigl([0,T];(E_1,E_0)_{\eta,\infty}\bigr),
\qquad 0\le\eta<\theta,
\]
and compact lower-order variants.  In a Hilbert triple $V\hookrightarrow\!\hookrightarrow H\hookrightarrow V^*$ this yields the sharp abstract threshold $\theta>1/2$ for compactness in $C([0,T];H)$.

On $\mathbb R^d$, local compactness is not global compactness because mass may escape by translation.  We prove a spatial tightness theorem in the same time-Morrey norm and then pass from compactness to defect-of-compactness.  For Hilbert-valued evolution classes we give a direct proof of cocompactness modulo translations.  Combined with a self-contained Hilbert translation-profile extraction, this yields profiles, Morrey--Bessel bounds on every time interval, and a remainder small in
\[
\mathcal M_t^{2,\mu}L_x^q,
\qquad \mu<\lambda,\quad 2<q<2^*.
\]
A uniform little-Morrey condition removes the loss $\mu<\lambda$.

The abstract language of cocompactness and profile decompositions is well developed; see Cwikel--Tintarev~\cite{CwikelTintarev2013} and Solimini--Tintarev~\cite{SoliminiTintarev2015}.  Those frameworks do not by themselves provide the concrete time-Morrey cocompact embedding above.  The analytic step here is the time-averaging estimate that converts translation-weak nullity in the Hilbert evolution space into strong $L_t^2L_x^q$ convergence, followed by a Morrey upgrade.

At the critical spatial endpoint, translations are no longer the whole defect group.  G\'erard's Sobolev profile theorem~\cite{Gerard1998} gives the precise translation--dilation decomposition of bounded homogeneous Hilbert-Sobolev sequences.  For heat and whole-space Stokes flows we show that parabolic dynamics lift G\'erard's critical Lebesgue-small remainder to the fully critical time-Morrey target
\[
\mathcal M_t^{2,\lambda}L_x^{2^*}.
\]
Thus the equation collapses the independent space/time defects of an arbitrary Morrey evolution to the balanced parabolic scale.

Finally, in three dimensions we compare with Gallagher's nonlinear Navier--Stokes profile theorem~\cite{Gallagher2001}.  We do not reprove the classical lifespan decomposition.  Instead we show that its two error mechanisms are automatically small in the full critical Morrey--Kato segment
\[
Y_p=\mathcal M_t^{p,1-p/4}L_x^6,
\qquad 2\le p<4.
\]
We also prove directly, in the Morrey forcing topology, that products of two orthogonal nonlinear profiles vanish.  For $2<p<4$ the corresponding Duhamel estimate is exactly the strong Morrey fractional-integral theorem of Adams~\cite{Adams1975}; the endpoint $p=2$ remains a valid profile-remainder topology but is not obtained from the same strong contraction estimate.

Spatial compact embeddings in smoothness Morrey scales constitute a separate literature; see, for example, \cite{GoncalvesHaroskeSkrzypczak2021}.  Our results concern a different object: Banach-valued \emph{evolution} spaces with Morrey control in time, exact trace-driven compactness, and evolution profile remainders measured in time-Morrey norms.

The paper is organized as follows.  Section~\ref{sec:prelim} records the evolution spaces and the trace input.  Sections~\ref{sec:mixed}--\ref{sec:lower-order} establish the continuous and compact embedding theory.  Section~\ref{sec:tightness} treats unbounded spatial domains.  Section~\ref{sec:profiles} develops Hilbert-valued cocompactness, profile decomposition, and the resulting trichotomy.  Section~\ref{sec:critical-heat-profiles} treats critical heat and Stokes profiles, and Section~\ref{sec:ns-morrey-profiles} gives the Navier--Stokes Morrey refinement.

\section{Morrey evolution spaces and trace input}\label{sec:prelim}

Let $I=(0,T)$ with $0<T<\infty$.  For a Banach space $X$, $1<p<\infty$, and $0\le\lambda<1$, define
\[
\|f\|_{\mathcal M^{p,\lambda}(I;X)}
:=
\sup_{J\subset I}
|J|^{-\lambda/p}
\left(\int_J\|f(t)\|_X^p\,dt\right)^{1/p},
\]
where the supremum is over nonempty intervals $J\subset I$.

\begin{definition}\label{def:evolution-space}
Let $E_0\hookrightarrow E_1$ continuously.  We set
\[
\mathbb W_M^{p,\lambda}(I;E_0,E_1)
:=
\left\{
 u\in\mathcal M^{p,\lambda}(I;E_0):
 u'\in\mathcal M^{p,\lambda}(I;E_1)
\right\},
\]
with graph norm
\[
\|u\|_{\mathbb W_M^{p,\lambda}}
=
\|u\|_{\mathcal M^{p,\lambda}(I;E_0)}
+
\|u'\|_{\mathcal M^{p,\lambda}(I;E_1)}.
\]
Throughout the genuinely Morrey regime we put
\[
\theta:=1-\frac{1-\lambda}{p},
\qquad 0<\lambda<1,
\]
and
\[
X_{\eta,q}:=(E_1,E_0)_{\eta,q},
\qquad 0<\eta<1,
\]
with $X_{0,q}=E_1$ by convention.  We abbreviate $X_\eta:=X_{\eta,\infty}$.
\end{definition}

The finite interval itself is admissible in the Morrey supremum, hence
\begin{equation}\label{eq:morrey-to-lp}
\|f\|_{L^p(I;X)}
\le T^{\lambda/p}\|f\|_{\mathcal M^{p,\lambda}(I;X)}.
\end{equation}
In particular every element of $\mathbb W_M^{p,\lambda}$ belongs to $W^{1,p}(I;E_1)$ after using $E_0\hookrightarrow E_1$, and therefore has an $E_1$-valued absolutely continuous representative.

We shall use the exact trace theorem proved in~\cite{ShahmurovTrace2026}.  We record precisely the two features needed here.

\begin{theorem}\label{thm:trace-input}
Let $E_0\hookrightarrow E_1$, $1<p<\infty$, and $0<\lambda<1$.  Then
\[
\operatorname{Tr}_0\mathbb W_M^{p,\lambda}(I;E_0,E_1)
=X_\theta=(E_1,E_0)_{\theta,\infty}.
\]
The trace map is bounded and onto.  Moreover, every $x\in X_\theta$ admits an extension $R x\in\mathbb W_M^{p,\lambda}(I;E_0,E_1)$ such that
\[
(Rx)(0)=x,
\qquad
\|Rx\|_{\mathbb W_M^{p,\lambda}}\le C\|x\|_{X_\theta}.
\]
Here $R$ denotes a fixed bounded extension assignment; linearity is not required for the compactness arguments below.
Moreover every $u\in\mathbb W_M^{p,\lambda}$ has an $E_1$-continuous representative satisfying
\begin{align}
\sup_{t\in[0,T]}\|u(t)\|_{X_\theta}
&\le C\|u\|_{\mathbb W_M^{p,\lambda}},\label{eq:trace-uniform}\\
\|u(t)-u(s)\|_{E_1}
&\le C|t-s|^\theta\|u'\|_{\mathcal M^{p,\lambda}(I;E_1)}.
\label{eq:e1-holder}
\end{align}
\end{theorem}

\begin{remark}\label{rem:trace-estimate-direct}
The H\"older estimate \eqref{eq:e1-holder} is immediate from H\"older's inequality and the Morrey bound:
\[
\|u(t)-u(s)\|_{E_1}
\le |t-s|^{1-1/p}
\left(\int_s^t\|u'(r)\|_{E_1}^p\,dr\right)^{1/p}
\le C|t-s|^\theta\|u'\|_{\mathcal M^{p,\lambda}}.
\]
The nontrivial input in Theorem~\ref{thm:trace-input} is the exact pointwise space $X_\theta$ and the bounded trace extension.  The extension used in~\cite{ShahmurovTrace2026} is obtained by an explicit dyadic construction: for $x\in X_\theta$ one chooses near-minimizing $K$-functional decompositions at the scales $T2^{-n}$ and linearly interpolates their strong components on the dyadic time intervals.  Both the extension and its derivative are then pointwise dominated by $Ct^{-(1-\lambda)/p}\|x\|_{X_\theta}$, and this scalar power belongs to $\mathcal M^{p,\lambda}(0,T)$.  Thus the compactness arguments below require neither complementability of $X_\theta$ nor a linear coretraction.  The existence of this bounded extension is what makes the compactness criterion in Section~\ref{sec:exact-C} necessary as well as sufficient.
\end{remark}

We also use the elementary interpolation inequality supplied by the reiteration theorem.  If $0\le\eta<\sigma<1$, then
\begin{equation}\label{eq:reiteration-ineq}
\|x\|_{X_\eta}
\le C
\|x\|_{E_1}^{1-\eta/\sigma}
\|x\|_{X_\sigma}^{\eta/\sigma},
\qquad x\in X_\sigma.
\end{equation}

\section{Continuous mixed embeddings}\label{sec:mixed}

The first consequence of the trace theorem is a full time/interpolation scale below the exact trace level.

\begin{theorem}\label{thm:mixed-embedding}
Let the assumptions of Theorem~\ref{thm:trace-input} hold.  For every $0\le\eta<\theta$,
\[
\mathbb W_M^{p,\lambda}(I;E_0,E_1)
\hookrightarrow
C^{0,\theta-\eta}([0,T];X_\eta).
\]
More precisely,
\[
\|u\|_{C^{0,\theta-\eta}([0,T];X_\eta)}
\le C_{p,\lambda,T,\eta}
\|u\|_{\mathbb W_M^{p,\lambda}}.
\]
\end{theorem}

\begin{proof}
The uniform $X_\theta$ bound follows from \eqref{eq:trace-uniform}.  Let $s,t\in[0,T]$ and set $v=u(t)-u(s)$.  Then
\[
\|v\|_{E_1}\le C|t-s|^\theta\|u\|_{\mathbb W_M^{p,\lambda}}
\]
by \eqref{eq:e1-holder}, while
\[
\|v\|_{X_\theta}
\le 2\sup_{r\in[0,T]}\|u(r)\|_{X_\theta}
\le C\|u\|_{\mathbb W_M^{p,\lambda}}.
\]
Apply \eqref{eq:reiteration-ineq} with $\sigma=\theta$.  We obtain
\[
\|u(t)-u(s)\|_{X_\eta}
\le C|t-s|^{\theta(1-\eta/\theta)}
\|u\|_{\mathbb W_M^{p,\lambda}}
=C|t-s|^{\theta-\eta}
\|u\|_{\mathbb W_M^{p,\lambda}}.
\]
The uniform $X_\eta$ bound follows from the continuous embedding $X_\theta\hookrightarrow X_\eta$.
\end{proof}

\begin{corollary}\label{cor:finite-q-mixed}
Let $0\le\eta<\theta$, $1\le q\le\infty$, and $0\le\beta<\theta-\eta$.  Then
\[
\mathbb W_M^{p,\lambda}(I;E_0,E_1)
\hookrightarrow
C^{0,\beta}([0,T];X_{\eta,q}).
\]
\end{corollary}

\begin{proof}
Choose $\sigma$ so that
\[
\eta<\sigma<\theta,
\qquad
\beta<\theta-\sigma.
\]
Theorem~\ref{thm:mixed-embedding} gives boundedness in
$C^{0,\theta-\sigma}([0,T];X_{\sigma,\infty})$.  The standard monotonicity theorem for real interpolation gives
\[
X_{\sigma,\infty}\hookrightarrow X_{\eta,q}
\]
whenever $\eta<\sigma$.  The conclusion follows.
\end{proof}

\begin{proposition}\label{prop:morrey-exponent-embedding}
Let $1\le q\le p<\infty$, $0\le\lambda,\mu<1$, and assume
\[
\frac{1-\mu}{q}\ge \frac{1-\lambda}{p}.
\]
Then, on a finite interval,
\[
\mathcal M^{p,\lambda}(I;X)
\hookrightarrow
\mathcal M^{q,\mu}(I;X)
\]
for every Banach space $X$.
\end{proposition}

\begin{proof}
For every interval $J\subset I$, H\"older's inequality gives
\[
\|f\|_{L^q(J;X)}
\le |J|^{1/q-1/p}\|f\|_{L^p(J;X)}.
\]
Hence
\[
|J|^{-\mu/q}\|f\|_{L^q(J;X)}
\le
|J|^{(1-\mu)/q-(1-\lambda)/p}
\|f\|_{\mathcal M^{p,\lambda}(I;X)}.
\]
The exponent is nonnegative, so the factor is bounded by the corresponding power of $T$.
\end{proof}

\section{Compactness in the same Morrey norm}\label{sec:same-morrey}

The following result is the Morrey-time counterpart of the classical Aubin--Lions lemma.  Its point is that the target is not merely $L^p(I;E)$: compactness holds in the same Morrey norm.

We first isolate a simple upgrading lemma.

\begin{lemma}\label{lem:lp-holder-to-uniform}
Let $X$ be a Banach space, $1\le p<\infty$, and $0<\alpha\le1$.  Suppose $(v_n)$ is Cauchy in $L^p(I;X)$ and
\[
\sup_n[v_n]_{C^{0,\alpha}([0,T];X)}<\infty.
\]
Then $(v_n)$ is Cauchy in $C([0,T];X)$.
\end{lemma}

\begin{proof}
It is enough to apply the argument to differences $w=v_n-v_m$.  Let $H$ be a common H\"older bound for all such differences.  If $\|w(t_0)\|_X\ge\varepsilon$, then
\[
\|w(t)\|_X\ge\frac\varepsilon2
\]
whenever $|t-t_0|\le (\varepsilon/(2H))^{1/\alpha}$.  Intersecting this interval with $[0,T]$ leaves a set of measure bounded from below by a positive number depending only on $\varepsilon,H,T$.  Therefore $\|w\|_{L^p(I;X)}$ is bounded from below by a positive constant.  This contradicts the $L^p$-Cauchy property.  Hence the sequence is uniformly Cauchy.
\end{proof}

\begin{theorem}\label{thm:same-morrey-aubin-lions}
Let
\[
E_0\hookc E\hookrightarrow E_1
\]
be Banach spaces.  Let $1<p<\infty$ and $0<\lambda<1$.  Then the canonical embedding
\[
\mathbb W_M^{p,\lambda}(I;E_0,E_1)
\hookrightarrow
\mathcal M^{p,\lambda}(I;E)
\]
is compact.
\end{theorem}

\begin{proof}
Let $(u_n)$ be bounded in $\mathbb W_M^{p,\lambda}(I;E_0,E_1)$.  By \eqref{eq:morrey-to-lp}, the sequence is bounded in $L^p(I;E_0)$ and $(u_n')$ is bounded in $L^p(I;E_1)$.  The classical Aubin--Lions--Simon theorem therefore yields a subsequence, not relabeled, which is Cauchy in $L^p(I;E)$.

Since $E\hookrightarrow E_1$, the same subsequence is Cauchy in $L^p(I;E_1)$.  On the other hand, \eqref{eq:e1-holder} gives a uniform $C^{0,\theta}([0,T];E_1)$ bound.  Lemma~\ref{lem:lp-holder-to-uniform} implies
\[
\|u_n-u_m\|_{C([0,T];E_1)}\longrightarrow0.
\]
Consequently,
\begin{equation}\label{eq:morrey-e1-conv}
\|u_n-u_m\|_{\mathcal M^{p,\lambda}(I;E_1)}
\le
T^{(1-\lambda)/p}
\|u_n-u_m\|_{C([0,T];E_1)}
\longrightarrow0.
\end{equation}

By Ehrling's lemma, for every $\varepsilon>0$ there exists $C_\varepsilon$ such that
\[
\|x\|_E
\le \varepsilon\|x\|_{E_0}+C_\varepsilon\|x\|_{E_1},
\qquad x\in E_0.
\]
Apply this inequality almost everywhere to $u_n-u_m$, take $L^p$ norms on an arbitrary interval $J\subset I$, multiply by $|J|^{-\lambda/p}$, and then take the supremum over $J$.  We obtain
\[
\|u_n-u_m\|_{\mathcal M^{p,\lambda}(I;E)}
\le
\varepsilon
\|u_n-u_m\|_{\mathcal M^{p,\lambda}(I;E_0)}
+C_\varepsilon
\|u_n-u_m\|_{\mathcal M^{p,\lambda}(I;E_1)}.
\]
The first factor is uniformly bounded, while the second tends to zero by \eqref{eq:morrey-e1-conv}.  First choose $\varepsilon$ small and then $n,m$ large.  Thus $(u_n)$ is Cauchy in $\mathcal M^{p,\lambda}(I;E)$, which proves compactness.  Notice that no reflexivity assumption on any of the three Banach spaces has been used: the only compactness input is the classical Aubin--Lions--Simon theorem for $E_0\hookc E\hookrightarrow E_1$, while $p>1$ is used in the Morrey H\"older estimate.
\end{proof}

\begin{corollary}\label{cor:weaker-morrey-compact}
Under the assumptions of Theorem~\ref{thm:same-morrey-aubin-lions}, let $1\le q\le p$ and $0\le\mu<1$ satisfy
\[
\frac{1-\mu}{q}\ge \frac{1-\lambda}{p}.
\]
Then
\[
\mathbb W_M^{p,\lambda}(I;E_0,E_1)
\hookc
\mathcal M^{q,\mu}(I;E).
\]
\end{corollary}

\begin{proof}
Combine Theorem~\ref{thm:same-morrey-aubin-lions} with Proposition~\ref{prop:morrey-exponent-embedding}.
\end{proof}

\begin{remark}\label{rem:stronger-than-classical}
Theorem~\ref{thm:same-morrey-aubin-lions} is strictly a statement about the Morrey topology.  Classical Aubin--Lions supplies only the intermediate $L^p(I;E)$ compactness used in the proof.  The Morrey derivative estimate gives uniform H\"older control in the weak space $E_1$; this is what allows the $L^p$ convergence to be upgraded and then transferred back to the full Morrey norm through Ehrling's lemma.
\end{remark}

\section{Exact compactness in \texorpdfstring{$C([0,T];E)$}{C([0,T];E)}}\label{sec:exact-C}

The preceding theorem uses the classical geometry $E_0\hookc E\hookrightarrow E_1$.  Compactness in $C([0,T];E)$ has a different and sharper criterion.

\begin{theorem}\label{thm:exact-C-criterion}
Let $E$ be a Banach space continuously embedded in $E_1$.  Then
\[
\mathbb W_M^{p,\lambda}(I;E_0,E_1)
\hookc C([0,T];E)
\]
if and only if
\[
X_\theta=(E_1,E_0)_{\theta,\infty}
\hookc E.
\]
\end{theorem}

\begin{proof}
Assume first that $X_\theta\hookc E\hookrightarrow E_1$.  Let $(u_n)$ be bounded in the Morrey evolution space.  By Theorem~\ref{thm:trace-input}, the set
\[
\{u_n(t):n\in\mathbb N\}
\]
is bounded in $X_\theta$ for each fixed $t$, hence relatively compact in $E$.

Ehrling's lemma for the compact embedding $X_\theta\hookc E\hookrightarrow E_1$ gives, for every $\varepsilon>0$,
\[
\|x\|_E
\le \varepsilon\|x\|_{X_\theta}+C_\varepsilon\|x\|_{E_1},
\qquad x\in X_\theta.
\]
Apply this to $u_n(t)-u_n(s)$.  The $X_\theta$ term is uniformly bounded by \eqref{eq:trace-uniform}, while the $E_1$ term tends to zero uniformly in $n$ as $t\to s$ by \eqref{eq:e1-holder}.  Thus $(u_n)$ is equicontinuous in $E$.  Banach-valued Arzel\`a--Ascoli yields relative compactness in $C([0,T];E)$.

Conversely, assume that the canonical embedding
\[
J:\mathbb W_M^{p,\lambda}(I;E_0,E_1)\to C([0,T];E)
\]
is compact.  Let $(x_n)$ be bounded in $X_\theta$.  By Theorem~\ref{thm:trace-input}, choose extensions $u_n=Rx_n$ with
\[
 u_n(0)=x_n,
 \qquad
 \sup_n\|u_n\|_{\mathbb W_M^{p,\lambda}}<\infty.
\]
Compactness of $J$ gives a subsequence converging in $C([0,T];E)$.  Evaluating at $t=0$ shows that the corresponding subsequence of $(x_n)$ converges in $E$.  In particular every element of $X_\theta$ belongs to $E$, and every bounded sequence in the $X_\theta$-unit ball has an $E$-convergent subsequence.  Since the natural inclusion $X_\theta\to E$ is linear, sequential relative compactness of its unit ball implies boundedness and compactness of that inclusion.  Hence $X_\theta\hookc E$.
\end{proof}

\begin{corollary}\label{cor:defect-lifting}
If $X_\theta\hookrightarrow E$ is not compact, then there exists a bounded sequence $(u_n)$ in $\mathbb W_M^{p,\lambda}(I;E_0,E_1)$ with no convergent subsequence in $C([0,T];E)$.
\end{corollary}

\begin{proof}
Choose a bounded sequence $(x_n)$ in $X_\theta$ with no convergent subsequence in $E$ and set $u_n=Rx_n$, where $R$ is a bounded extension assignment from Theorem~\ref{thm:trace-input}.  Then $(u_n)$ is bounded in the evolution space and $u_n(0)=x_n$.  Convergence in $C([0,T];E)$ would force convergence of the traces in $E$, a contradiction.
\end{proof}

\begin{remark}\label{rem:defect-interpretation}
Corollary~\ref{cor:defect-lifting} is a simple defect-of-compactness principle: every noncompact bounded trace sequence can be lifted to a noncompact bounded evolution sequence.  In critical Sobolev settings the trace defect may be produced by translation, dilation, or concentration.  Section~\ref{sec:profiles} gives a full translation profile theorem in the Hilbert-valued subcritical Morrey regime; the genuinely critical dilation problem requires a larger anisotropic dislocation group.
\end{remark}

\section{Compact lower-order embeddings}\label{sec:lower-order}

We now assume that the two endpoint spaces themselves form a compact couple.

\begin{lemma}\label{lem:compact-interpolation-gap}
Assume $E_0\hookc E_1$.  Let
\[
0\le\eta<\sigma<1,
\qquad 1\le q,r\le\infty.
\]
Then the natural embedding
\[
(E_1,E_0)_{\sigma,r}
\hookc
(E_1,E_0)_{\eta,q}
\]
is compact.
\end{lemma}

\begin{proof}
We give a direct proof that also covers the fine-index endpoint $q=\infty$.  Put
\[
X_{\sigma,r}:=(E_1,E_0)_{\sigma,r}.
\]
Since $X_{\sigma,r}\hookrightarrow X_{\sigma,\infty}$, a bounded set $B\subset X_{\sigma,r}$ satisfies
\[
K(t,x;E_1,E_0)\le C_B t^\sigma,
\qquad x\in B,
\]
for all $t>0$.  Fix $0<t<1$.  For every $x\in B$ choose
\[
x=a_t(x)+b_t(x),\qquad a_t(x)\in E_1,\quad b_t(x)\in E_0,
\]
with
\[
\|a_t(x)\|_{E_1}+t\|b_t(x)\|_{E_0}\le 2C_Bt^\sigma.
\]
Thus the $a_t$-part is uniformly $O(t^\sigma)$ in $E_1$, while the $b_t$-part is bounded in $E_0$ for fixed $t$.  Since $E_0\hookc E_1$, the set of $b_t$-parts is relatively compact in $E_1$.  First choosing $t$ small and then a finite $E_1$-net for the $b_t$-parts shows that $B$ is totally bounded in $E_1$.  Therefore
\[
X_{\sigma,r}\hookc E_1.
\]

Let $(x_n)$ be bounded in $X_{\sigma,r}$ and, after extraction, suppose $x_n\to x$ in $E_1$.  For every $t>0$, lower semicontinuity of the $K$-functional under $E_1$ convergence gives
\[
K(t,x;E_1,E_0)\le\liminf_{n\to\infty}K(t,x_n;E_1,E_0).
\]
Fatou's lemma for $r<\infty$ and the corresponding supremum estimate for $r=\infty$ show that $x\in X_{\sigma,r}$.  Hence $z_n:=x_n-x$ is bounded in $X_{\sigma,r}$ and therefore in $X_{\sigma,\infty}$.

If $q=\infty$ and $0<\eta<\sigma$, the elementary interpolation inequality yields
\[
\|z_n\|_{(E_1,E_0)_{\eta,\infty}}
\le C
\|z_n\|_{E_1}^{1-\eta/\sigma}
\|z_n\|_{(E_1,E_0)_{\sigma,\infty}}^{\eta/\sigma}
\longrightarrow0.
\]
If $q<\infty$, choose $\eta'$ with $\eta<\eta'<\sigma$.  The preceding estimate gives
\[
z_n\to0\quad\text{in }(E_1,E_0)_{\eta',\infty},
\]
and the strict-smoothness embedding
\[
(E_1,E_0)_{\eta',\infty}\hookrightarrow(E_1,E_0)_{\eta,q}
\]
gives the desired convergence.  When $\eta=0$, the target is $E_1$ by convention, and convergence is already known.  Thus the embedding is compact for every $1\le q,r\le\infty$.
\end{proof}

\begin{theorem}\label{thm:compact-lower-order}
Assume $E_0\hookc E_1$.  Let $0\le\eta<\theta$, $1\le q\le\infty$, and $0\le\beta<\theta-\eta$.  Then
\[
\mathbb W_M^{p,\lambda}(I;E_0,E_1)
\hookc
C^{0,\beta}([0,T];X_{\eta,q}).
\]
\end{theorem}

\begin{proof}
Choose $\sigma$ with
\[
\eta<\sigma<\theta,
\qquad
\beta<\theta-\sigma.
\]
By Theorem~\ref{thm:mixed-embedding}, bounded subsets of the evolution space are bounded in
\[
C^{0,\theta-\sigma}([0,T];X_{\sigma,\infty}).
\]
Lemma~\ref{lem:compact-interpolation-gap} gives
\[
X_{\sigma,\infty}\hookc X_{\eta,q}.
\]
Therefore the family is pointwise relatively compact and equicontinuous in $X_{\eta,q}$, hence relatively compact in $C([0,T];X_{\eta,q})$ by Arzel\`a--Ascoli.

Finally, if $v_n\to0$ uniformly in $X_{\eta,q}$ and $(v_n)$ is uniformly bounded in $C^{0,\alpha}([0,T];X_{\eta,q})$ for some $\alpha>\beta$, then the elementary H\"older interpolation estimate
\[
[v_n]_{C^{0,\beta}}
\le C
\|v_n\|_{C}^{1-\beta/\alpha}
[v_n]_{C^{0,\alpha}}^{\beta/\alpha}
\]
shows convergence in $C^{0,\beta}$.  Taking $\alpha=\theta-\sigma$ completes the proof.
\end{proof}

\subsection{A sharp Hilbert-triple threshold}

The previous theorem has a particularly transparent consequence for the standard Gelfand triple used in weak formulations of parabolic equations.

\begin{corollary}\label{cor:hilbert-triple}
Let
\[
V\hookc H\hookrightarrow V^*
\]
be a Hilbert triple and assume the standard interpolation identity
\[
H=(V^*,V)_{1/2,2}
\]
with equivalent norms.  Let $1<p<\infty$, $0<\lambda<1$, and
\[
\theta=1-\frac{1-\lambda}{p}.
\]
Then every bounded subset of
\[
\mathbb W_M^{p,\lambda}(I;V,V^*)
\]
is relatively compact in $\mathcal M^{p,\lambda}(I;H)$.  If in addition
\[
\theta>\frac12,
\]
then it is relatively compact in
\[
C^{0,\beta}([0,T];H)
\qquad
\text{for every }0\le\beta<\theta-\frac12.
\]
\end{corollary}

\begin{proof}
The Morrey compactness follows from Theorem~\ref{thm:same-morrey-aubin-lions}.  For the time-continuous conclusion apply Theorem~\ref{thm:compact-lower-order} with $E_0=V$, $E_1=V^*$, $\eta=1/2$, and $q=2$.
\end{proof}

The condition $\theta>1/2$ is equivalent to
\[
\frac{1-\lambda}{p}<\frac12.
\]
Thus genuine Morrey control can move an energy class above the classical critical trace level.  For instance, when $p=2$, every $\lambda>0$ yields compactness in $C([0,T];H)$, whereas the classical $\lambda=0$ energy class sits exactly at the critical level.

\begin{proposition}\label{prop:hilbert-sharpness}
The threshold $\theta>1/2$ in Corollary~\ref{cor:hilbert-triple} is sharp at the level of abstract Hilbert triples.
\end{proposition}

\begin{proof}
Let
\[
H=\ell^2,
\qquad
V=\left\{x=(x_n):\sum_{n\ge1}n^2|x_n|^2<\infty\right\},
\]
with the natural norm, and let $V^*$ be the weighted space with norm
\[
\|x\|_{V^*}^2=\sum_{n\ge1}n^{-2}|x_n|^2.
\]
Then $V\hookc H\hookrightarrow V^*$.  For the coordinate vector $e_n$, the $K$-functional of the couple $(V^*,V)$ satisfies
\[
K(t,e_n;V^*,V)=\min\{n^{-1},tn\}.
\]
Consequently
\[
\|e_n\|_{(V^*,V)_{\theta,\infty}}
\asymp n^{2\theta-1}.
\]
Set
\[
x_n=n^{1-2\theta}e_n.
\]
Then $(x_n)$ is bounded in $(V^*,V)_{\theta,\infty}$.  If $\theta<1/2$, its $H$ norm tends to infinity, so the trace space does not even embed continuously into $H$.  If $\theta=1/2$, then $x_n=e_n$ is bounded in the trace space and has no convergent subsequence in $H$.  By Corollary~\ref{cor:defect-lifting}, these trace defects lift to bounded sequences in the Morrey evolution class.  Thus no abstract $C([0,T];H)$ compactness theorem can hold at or below the threshold.
\end{proof}

\section{Tightness on unbounded domains}\label{sec:tightness}

Theorem~\ref{thm:same-morrey-aubin-lions} is local in space when the spatial compact embedding is local.  On $\mathbb R^d$, translation destroys global compactness.  The next result shows that the missing ingredient is exactly a spatial tightness condition.

\begin{theorem}\label{thm:tightness-Rd}
Let $1<r<\infty$, $1<p<\infty$, and $0<\lambda<1$.  Let $(u_n)$ satisfy
\[
\sup_n\left(
\|u_n\|_{\mathcal M^{p,\lambda}(I;W^{1,r}(\R^d))}
+
\|u_n'\|_{\mathcal M^{p,\lambda}(I;W^{-1,r}(\R^d))}
\right)<\infty.
\]
Assume in addition the uniform Morrey tightness condition
\[
\lim_{R\to\infty}
\sup_n
\|u_n\|_{\mathcal M^{p,\lambda}(I;L^r(\R^d\setminus B_R))}
=0.
\]
Then $(u_n)$ has a subsequence converging strongly in
\[
\mathcal M^{p,\lambda}(I;L^r(\R^d)).
\]
\end{theorem}

\begin{proof}
Fix $R>0$.  Restriction to $B_R$ is bounded from $W^{1,r}(\R^d)$ to $W^{1,r}(B_R)$.  The derivative restricts boundedly to $W^{-1,r}(B_R)$ by duality, using zero extension of test functions from $W_0^{1,r'}(B_R)$ to $W^{1,r'}(\R^d)$.  Since
\[
W^{1,r}(B_R)\hookc L^r(B_R)\hookrightarrow W^{-1,r}(B_R),
\]
Theorem~\ref{thm:same-morrey-aubin-lions} gives relative compactness in
\[
\mathcal M^{p,\lambda}(I;L^r(B_R)).
\]
Applying this for $R=1,2,3,\dots$ and taking a diagonal subsequence, we obtain a subsequence that is Cauchy in the local Morrey norm on every ball.

Let $\varepsilon>0$.  Choose $R$ so large that the tightness bound is at most $\varepsilon$ for every $n$.  For the diagonal subsequence and sufficiently large $n,m$,
\[
\|u_n-u_m\|_{\mathcal M^{p,\lambda}(I;L^r(B_R))}<\varepsilon.
\]
By the triangle inequality in $L^r(\R^d)$ and then in the Morrey norm,
\begin{align*}
\|u_n-u_m\|_{\mathcal M^{p,\lambda}(I;L^r(\R^d))}
&\le
\|u_n-u_m\|_{\mathcal M^{p,\lambda}(I;L^r(B_R))}\\
&\quad+
\|u_n\|_{\mathcal M^{p,\lambda}(I;L^r(\R^d\setminus B_R))}\\
&\quad+
\|u_m\|_{\mathcal M^{p,\lambda}(I;L^r(\R^d\setminus B_R))}
<3\varepsilon.
\end{align*}
Hence the subsequence is Cauchy in the global Morrey space.
\end{proof}

\begin{example}\label{ex:translation-defect}
The tightness hypothesis cannot be dropped on $\R^d$.  Let
\[
\phi\in C_c^1(0,T),
\qquad
\psi\in C_c^\infty(\R^d),
\]
be nonzero and set
\[
u_n(t,x)=\phi(t)\psi(x-ne_1).
\]
Translation invariance gives uniform bounds in
\[
\mathcal M^{p,\lambda}(I;W^{1,r}(\R^d))
\]
and for the time derivatives in
\[
\mathcal M^{p,\lambda}(I;W^{-1,r}(\R^d)).
\]
For widely separated indices the spatial supports are disjoint, so the sequence has no Cauchy subsequence in
$\mathcal M^{p,\lambda}(I;L^r(\R^d))$.  Its failure is exactly loss of tightness at spatial infinity.
\end{example}

\section{Hilbert-valued cocompactness and profile decomposition}\label{sec:profiles}

The compactness results above identify what happens when the spatial embedding is compact or when tails are tight.  We now describe the complementary Hilbert-valued situation in which compactness fails only through spatial translations.  Throughout this section let $d\ge3$, let $I=(0,T)$, and define
\[
\begin{aligned}
\mathcal X_\lambda(I):=\{u:\;&u\in\mathcal M^{2,\lambda}(I;H^1(\mathbb R^d)),\\
&u_t\in\mathcal M^{2,\lambda}(I;H^{-1}(\mathbb R^d))\},
\qquad 0<\lambda<1.
\end{aligned}
\]
We also use the Hilbert evolution space
\[
\mathcal E(I)
:=L^2(I;H^1(\mathbb R^d))
\cap H^1(I;H^{-1}(\mathbb R^d)),
\]
with its natural Hilbert norm.  Since $I$ is finite,
\[
\mathcal X_\lambda(I)\hookrightarrow\mathcal E(I)
\]
continuously.  Spatial translations
\[
(T_yu)(t,x):=u(t,x-y),\qquad y\in\mathbb R^d,
\]
act isometrically on both spaces.

The first lemma is a simple but useful principle: a uniform Morrey bound upgrades ordinary $L^2$ convergence to convergence in every strictly weaker Morrey scale.

\begin{lemma}\label{lem:morrey-downgrade}
Let $Y$ be a Banach space, $0\le\mu<\lambda<1$, and let $(f_n)$ be bounded in
$\mathcal M^{2,\lambda}(I;Y)$.  If
\[
\|f_n\|_{L^2(I;Y)}\longrightarrow0,
\]
then
\[
\|f_n\|_{\mathcal M^{2,\mu}(I;Y)}\longrightarrow0.
\]
\end{lemma}

\begin{proof}
Let
\[
M:=\sup_n\|f_n\|_{\mathcal M^{2,\lambda}(I;Y)}.
\]
For an interval $J\subset I$ with $|J|\le\delta$,
\[
|J|^{-\mu}\int_J\|f_n(t)\|_Y^2\,dt
\le M^2|J|^{\lambda-\mu}
\le M^2\delta^{\lambda-\mu}.
\]
For $|J|>\delta$,
\[
|J|^{-\mu}\int_J\|f_n(t)\|_Y^2\,dt
\le \delta^{-\mu}\|f_n\|_{L^2(I;Y)}^2.
\]
Given $\varepsilon>0$, first choose $\delta$ so that the first quantity is below $\varepsilon^2$, and then choose $n$ so large that the second is below $\varepsilon^2$.  Taking the supremum over $J$ proves the claim.
\end{proof}

We next record the spatial cocompactness input in a form that we shall use directly.  The statement is classical, but the short proof is included because the time-Morrey argument below depends on its precise translation formulation.

\begin{lemma}\label{lem:spatial-translation-cocompact}
Let $d\ge3$ and $2<q<2^*=2d/(d-2)$.  Suppose $(f_n)$ is bounded in $H^1(\mathbb R^d)$ and
\[
T_{y_n}f_n\rightharpoonup0
\quad\text{weakly in }H^1(\mathbb R^d)
\]
for every sequence $(y_n)\subset\mathbb R^d$.  Then
\[
\|f_n\|_{L^q(\mathbb R^d)}\longrightarrow0.
\]
\end{lemma}

\begin{proof}
Cover $\mathbb R^d$ by unit cubes with uniformly bounded overlap after a fixed enlargement.  The local Gagliardo--Nirenberg inequality gives
\[
\|f\|_{L^{2+4/d}(Q)}^{2+4/d}
\le C\|f\|_{L^2(2Q)}^{4/d}\|f\|_{H^1(2Q)}^2.
\]
After summing over the cubes,
\[
\|f\|_{L^{2+4/d}}^{2+4/d}
\le C
\left(\sup_{y\in\mathbb R^d}\|f\|_{L^2(B_2(y))}\right)^{4/d}
\|f\|_{H^1}^2.
\]
If the local $L^2$ supremum did not tend to zero, there would be centers $y_n$ and $c>0$ such that
\[
\|f_n\|_{L^2(B_2(y_n))}\ge c.
\]
A subsequence of $T_{y_n}f_n$ converges weakly in $H^1$; Rellich compactness on $B_2$ then makes the convergence strong in local $L^2$, so the weak limit is nonzero.  This contradicts the hypothesis.  Hence $f_n\to0$ in $L^{2+4/d}$.  Interpolation with the uniform $L^2$ and $L^{2^*}$ bounds gives convergence in every $L^q$, $2<q<2^*$.
\end{proof}

Abstract formulations of this cocompactness principle and its interpolation properties can be found in \cite{CwikelTintarev2013,SoliminiTintarev2015}.

\begin{proposition}\label{prop:evolution-cocompact-L2}
Let $2<q<2^*$.  Suppose $(u_n)$ is bounded in $\mathcal E(I)$ and satisfies
\[
T_{y_n}u_n\rightharpoonup0
\quad\text{weakly in }\mathcal E(I)
\]
for every sequence $(y_n)\subset\mathbb R^d$.  Then
\[
\|u_n\|_{L^2(I;L^q(\mathbb R^d))}\longrightarrow0.
\]
\end{proposition}

\begin{proof}
Fix a mesh size $h>0$ and partition $I$ into finitely many intervals $I_k$ of length at most $h$.  Let $P_hu$ be the piecewise constant time average
\[
(P_hu)(t):=\frac1{|I_k|}\int_{I_k}u(s)\,ds,
\qquad t\in I_k.
\]
The one-dimensional Poincar\'e inequality in the Hilbert space $H^{-1}$ gives
\[
\|u-P_hu\|_{L^2(I;H^{-1})}
\le Ch\|u_t\|_{L^2(I;H^{-1})}.
\]
Also
\[
\|u-P_hu\|_{L^2(I;H^1)}
\le2\|u\|_{L^2(I;H^1)}.
\]
Using the Hilbert-scale inequality
\[
\|v\|_{L^2_x}^2
\le C\|v\|_{H^{-1}_x}\|v\|_{H^1_x}
\]
and Cauchy--Schwarz in time, we obtain, uniformly on bounded subsets of $\mathcal E(I)$,
\[
\|u-P_hu\|_{L^2(I;L^2)}
\le C h^{1/2}.
\]
Let
\[
a=d\left(\frac12-\frac1q\right)\in(0,1).
\]
The Gagliardo--Nirenberg inequality
\[
\|v\|_{L^q}
\le C\|v\|_{L^2}^{1-a}\|v\|_{H^1}^{a}
\]
and H\"older's inequality in time then give
\[
\|u-P_hu\|_{L^2(I;L^q)}
\le C h^{(1-a)/2}
\]
uniformly for $u$ in a fixed bounded subset of $\mathcal E(I)$.

For fixed $h$ and each cell $I_k$, put
\[
a_{n,k}:=\frac1{|I_k|}\int_{I_k}u_n(s)\,ds\in H^1(\mathbb R^d).
\]
The sequence $(a_{n,k})_n$ is bounded in $H^1$.  Moreover, for every sequence $(y_n)$,
\[
T_{y_n}a_{n,k}\rightharpoonup0
\quad\text{weakly in }H^1,
\]
because time averaging is a bounded linear map from $\mathcal E(I)$ to $H^1$.  Lemma~\ref{lem:spatial-translation-cocompact} therefore yields
\[
\|a_{n,k}\|_{L^q}\longrightarrow0
\]
for each fixed $k$.  Since the partition has finitely many cells,
\[
\|P_hu_n\|_{L^2(I;L^q)}^2
=\sum_k |I_k|\|a_{n,k}\|_{L^q}^2
\longrightarrow0.
\]
Consequently,
\[
\limsup_{n\to\infty}\|u_n\|_{L^2(I;L^q)}
\le C h^{(1-a)/2}.
\]
Letting $h\downarrow0$ proves the assertion.
\end{proof}

\begin{theorem}\label{thm:morrey-cocompact}
Let $0\le\mu<\lambda<1$ and $2<q<2^*$.  Suppose $(u_n)$ is bounded in $\mathcal X_\lambda(I)$ and
\[
T_{y_n}u_n\rightharpoonup0
\quad\text{weakly in }\mathcal E(I)
\]
for every sequence $(y_n)\subset\mathbb R^d$.  Then
\[
\|u_n\|_{\mathcal M^{2,\mu}(I;L^q(\mathbb R^d))}
\longrightarrow0.
\]
\end{theorem}

\begin{proof}
By Proposition~\ref{prop:evolution-cocompact-L2},
\[
\|u_n\|_{L^2(I;L^q)}\to0.
\]
Sobolev embedding gives the uniform bound
\[
\sup_n\|u_n\|_{\mathcal M^{2,\lambda}(I;L^q)}<\infty.
\]
Lemma~\ref{lem:morrey-downgrade}, with $Y=L^q(\mathbb R^d)$, completes the proof.
\end{proof}

We now isolate the Hilbert translation-profile extraction needed below.  It is the standard greedy dislocation argument, included here to keep the Morrey remainder theorem independent of an abstract profile-decomposition black box.

\begin{lemma}\label{lem:translation-dislocation-E}
Let
\[
\mathcal E(I)=L^2(I;H^1(\mathbb R^d))\cap H^1(I;H^{-1}(\mathbb R^d)).
\]
Then spatial translations act unitarily on $\mathcal E(I)$.  Moreover:
\begin{enumerate}[label=\textup{(\roman*)}]
\item if $|y_n|\to\infty$, then $T_{y_n}U\rightharpoonup0$ weakly in $\mathcal E(I)$ for every $U\in\mathcal E(I)$;
\item if $(y_n)$ is bounded, then after extraction $y_n\to y$ and $T_{y_n}U\to T_yU$ strongly in $\mathcal E(I)$ for every $U\in\mathcal E(I)$.
\end{enumerate}
\end{lemma}

\begin{proof}
Unitarity follows from translation invariance of the $H^1$ and $H^{-1}$ norms, applied also to the time derivative.  For \textup{(i)}, first take $U$ and a test vector $\Phi$ smooth in time and compactly supported in the spatial variable.  Each of the Hilbert pairings defining the $\mathcal E(I)$ inner product tends to zero when the spatial translate leaves every compact set.  Density of such test functions in $\mathcal E(I)$ and unitarity give the assertion for arbitrary $U$ and arbitrary test vectors.  For \textup{(ii)}, take a convergent subsequence $y_n\to y$.  Strong continuity of translations in $H^1$ and $H^{-1}$, followed by dominated convergence in time, yields strong convergence in both Hilbert components of $\mathcal E(I)$.
\end{proof}

\begin{lemma}\label{lem:hilbert-translation-profile}
Let $(u_n)$ be bounded in the Hilbert space $\mathcal E(I)$.  After passing to a subsequence, there exist profiles $U^j\in\mathcal E(I)$, translations $y_n^j\in\mathbb R^d$, and remainders $r_n^J$ such that for every finite $J$,
\[
u_n=\sum_{j=1}^J T_{y_n^j}U^j+r_n^J,
\]
with
\[
|y_n^j-y_n^k|\to\infty\qquad(j\ne k),
\]
and
\[
\|u_n\|_{\mathcal E(I)}^2
=\sum_{j=1}^J\|U^j\|_{\mathcal E(I)}^2
+\|r_n^J\|_{\mathcal E(I)}^2+o_n(1).
\]
Moreover the residual concentration vanishes in the sense that
\[
\lim_{J\to\infty}\eta(r^J)=0,
\]
where
\[
\eta(v):=\sup\left\{\|W\|_{\mathcal E(I)}:\
\begin{array}{l}
T_{-z_n}v_n\rightharpoonup W\text{ along a subsequence},\\
(z_n)\subset\mathbb R^d
\end{array}\right\}.
\]
\end{lemma}

\begin{proof}
Set $r_n^0=u_n$.  If $\eta(r^0)=0$, there is nothing to extract.  Otherwise, after taking a subsequence, choose translations $y_n^1$ and a nonzero profile $U^1$ such that
\[
T_{-y_n^1}r_n^0\rightharpoonup U^1,
\qquad
\|U^1\|_{\mathcal E(I)}\ge\frac12\eta(r^0),
\]
and put $r_n^1=r_n^0-T_{y_n^1}U^1$.  Then
\[
T_{-y_n^1}r_n^1\rightharpoonup0,
\]
and Hilbert orthogonality gives
\[
\|r_n^0\|_{\mathcal E(I)}^2
=\|U^1\|_{\mathcal E(I)}^2+\|r_n^1\|_{\mathcal E(I)}^2+o_n(1).
\]
Iterate this construction.  At step $J+1$ choose a profile whose norm is at least $\eta(r^J)/2$.  If a new translation $y_n^{J+1}$ stayed at bounded distance from one of the previous translations, then after extraction the two translations would differ by a fixed limit and the previously imposed weak-null condition would force the new profile to vanish.  Hence distinct translation parameters diverge.

Iterating the Pythagorean identity gives
\[
\sum_{j\ge1}\|U^j\|_{\mathcal E(I)}^2
\le\limsup_n\|u_n\|_{\mathcal E(I)}^2.
\]
If $\eta(r^J)$ did not tend to zero, the construction would extract infinitely many profiles with norms bounded below by one fixed positive number, contradicting the preceding square-summability.  A standard diagonal choice of the nested subsequences makes all finite decompositions valid on one final subsequence.
\end{proof}

This is the Hilbert specialization of the general dislocation philosophy; compare \cite{SoliminiTintarev2015}.  The new point below is that the time-Morrey bounds survive the extraction and force a Morrey-small remainder.

\begin{theorem}\label{thm:morrey-translation-profiles}
Let $(u_n)$ be bounded in $\mathcal X_\lambda(I)$.  After passing to a subsequence, there exist profiles
\[
U^j\in\mathcal X_\lambda(I),\qquad j\ge1,
\]
and translations $y_n^j\in\mathbb R^d$ such that, for every finite $J$,
\[
u_n(t,x)
=\sum_{j=1}^J U^j(t,x-y_n^j)+r_n^J(t,x),
\]
and the following properties hold.

\begin{enumerate}[label=\textup{(\roman*)}]
\item For $j\ne k$,
\[
|y_n^j-y_n^k|\longrightarrow\infty.
\]
\item The Hilbert evolution energy decouples:
\[
\|u_n\|_{\mathcal E(I)}^2
=\sum_{j=1}^J\|U^j\|_{\mathcal E(I)}^2
+\|r_n^J\|_{\mathcal E(I)}^2+o_n(1).
\]
In particular,
\[
\sum_{j\ge1}\|U^j\|_{\mathcal E(I)}^2
\le\limsup_{n\to\infty}\|u_n\|_{\mathcal E(I)}^2.
\]
\item For every time interval $K\subset I$,
\begin{align*}
\sum_{j\ge1}\|U^j\|_{L^2(K;H^1)}^2
&\le\liminf_{n\to\infty}\|u_n\|_{L^2(K;H^1)}^2,\\
\sum_{j\ge1}\|U_t^j\|_{L^2(K;H^{-1})}^2
&\le\liminf_{n\to\infty}\|(u_n)_t\|_{L^2(K;H^{-1})}^2.
\end{align*}
Consequently the profile energy densities obey the Morrey--Bessel bounds
\begin{align*}
\sup_{K\subset I}|K|^{-\lambda}
\sum_{j\ge1}\|U^j\|_{L^2(K;H^1)}^2
&\le M_0^2,\\
\sup_{K\subset I}|K|^{-\lambda}
\sum_{j\ge1}\|U_t^j\|_{L^2(K;H^{-1})}^2
&\le M_1^2,
\end{align*}
where
\[
M_0:=\limsup_n\|u_n\|_{\mathcal M^{2,\lambda}(I;H^1)},
\qquad
M_1:=\limsup_n\|(u_n)_t\|_{\mathcal M^{2,\lambda}(I;H^{-1})}.
\]
\item For every $0\le\mu<\lambda$ and every $2<q<2^*$,
\[
\boxed{
\lim_{J\to\infty}\limsup_{n\to\infty}
\|r_n^J\|_{\mathcal M^{2,\mu}(I;L^q)}=0.
}
\]
\end{enumerate}
\end{theorem}

\begin{proof}
Apply Lemma~\ref{lem:hilbert-translation-profile} in $\mathcal E(I)$.  This gives profiles $U^j\in\mathcal E(I)$, pairwise divergent translation parameters, the decomposition, the Hilbert energy identity, and remainders whose residual translation concentration tends to zero.  Because each profile is a weak limit in $\mathcal E(I)$ of spatial translates of a bounded residual sequence, lower semicontinuity on every fixed interval shows that $U^j$ belongs to $\mathcal X_\lambda(I)$.

We record the localization needed below.  Restriction from $I$ to a fixed interval $K\subset I$ is a bounded linear map on both Hilbert components of $\mathcal E(I)$.  Hence all weak convergence and residual orthogonality relations used in the extraction remain valid after restriction to $K$.  For a fixed finite number of profiles, the divergent spatial translations make distinct profiles asymptotically orthogonal in both $L^2(K;H^1)$ and $L^2(K;H^{-1})$ for the time derivatives.  The residual is weakly orthogonal to each extracted profile after undoing its translation.  Hence the usual Pythagorean expansion holds on $K$ separately in the two Hilbert components.  Dropping the nonnegative residual term and passing to the limit gives the two inequalities in \textup{(iii)}.  Taking the supremum after multiplying by $|K|^{-\lambda}$ yields the Morrey--Bessel bounds.

It remains to prove \textup{(iv)}.  Fix
\[
0\le\mu<\nu<\lambda.
\]
For a finite profile sum
\[
S_n^J:=\sum_{j=1}^J T_{y_n^j}U^j,
\]
the diagonal part of the squared $H^1$ norm is controlled by the first Morrey--Bessel bound.  Every cross term tends to zero in $L^1(I)$ as $n\to\infty$ because the relative translations diverge.  The cross terms are bounded in the scalar Morrey space $\mathcal M^{1,\lambda}(I)$ by Cauchy--Schwarz and the Morrey bounds of the fixed profiles.  Lemma~\ref{lem:morrey-downgrade}, with the obvious $L^1$ version, therefore shows that these cross terms vanish in $\mathcal M^{1,\nu}(I)$.  Thus
\[
\limsup_{n\to\infty}
\|S_n^J\|_{\mathcal M^{2,\nu}(I;H^1)}
\le C_{T,\lambda,\nu}M_0,
\]
with a constant independent of $J$.  The same argument for the time derivatives gives
\[
\limsup_{n\to\infty}
\|(S_n^J)_t\|_{\mathcal M^{2,\nu}(I;H^{-1})}
\le C_{T,\lambda,\nu}M_1.
\]
Consequently the remainders $r_n^J=u_n-S_n^J$ are bounded in $\mathcal X_\nu(I)$ uniformly with respect to $J$ after taking the outer $\limsup_n$.

Suppose now that the conclusion in \textup{(iv)} fails for the chosen $\mu$ and $q$.  Then there exist $\varepsilon>0$ and integers $J_m\to\infty$.  Using the outer $\limsup_n$ in the preceding uniform bounds, choose $n_m\to\infty$ so large that
\[
\|r_{n_m}^{J_m}\|_{\mathcal M^{2,\mu}(I;L^q)}\ge\varepsilon,
\]
the diagonal sequence $(r_{n_m}^{J_m})$ is bounded in $\mathcal X_\nu(I)$, and the first $J_m$ translated profiles are almost orthogonal in $\mathcal E(I)$ to accuracy $2^{-m}$.  More precisely, after increasing $n_m$ if necessary we may require
\[
\left|
\left\langle T_{y_{n_m}^j}U^j,T_{y_{n_m}^k}U^k\right\rangle_{\mathcal E(I)}
\right|
\le \frac{2^{-m}}{J_m^2},
\qquad 1\le j<k\le J_m.
\]

We claim that the diagonal remainder is translation-weakly null in $\mathcal E(I)$.  Let $(z_m)$ be arbitrary and suppose, after extraction, that
\[
T_{z_m}r_{n_m}^{J_m}\rightharpoonup W
\quad\text{in }\mathcal E(I).
\]
Fix $K$.  For $m$ large enough $J_m>K$, and write
\[
r_{n_m}^{J_m}=r_{n_m}^{K}-S_{m,K},
\qquad
S_{m,K}:=\sum_{j=K+1}^{J_m}T_{y_{n_m}^j}U^j.
\]
The almost-orthogonality choice gives
\[
\limsup_{m\to\infty}\|S_{m,K}\|_{\mathcal E(I)}^2
\le
\sum_{j>K}\|U^j\|_{\mathcal E(I)}^2.
\]
After a further subsequence, $T_{z_m}r_{n_m}^{K}$ has a weak limit $W_K$; by the definition of $\eta(r^K)$,
\[
\|W_K\|_{\mathcal E(I)}\le\eta(r^K).
\]
Since
\[
T_{z_m}S_{m,K}=T_{z_m}r_{n_m}^{K}-T_{z_m}r_{n_m}^{J_m},
\]
weak lower semicontinuity gives
\[
\|W_K-W\|_{\mathcal E(I)}
\le
\left(\sum_{j>K}\|U^j\|_{\mathcal E(I)}^2\right)^{1/2}.
\]
Consequently
\[
\|W\|_{\mathcal E(I)}
\le
\eta(r^K)
+
\left(\sum_{j>K}\|U^j\|_{\mathcal E(I)}^2\right)^{1/2}.
\]
Both terms tend to zero as $K\to\infty$.  Hence $W=0$.  Since $(z_m)$ was arbitrary, the diagonal sequence is translation-weakly null.

Theorem~\ref{thm:morrey-cocompact}, applied with source exponent $\nu$ and target exponent $\mu$, now gives
\[
\|r_{n_m}^{J_m}\|_{\mathcal M^{2,\mu}(I;L^q)}\longrightarrow0,
\]
contradicting the lower bound $\varepsilon$.  This proves \textup{(iv)}.
\end{proof}

There is a natural condition under which the loss in the time-Morrey exponent is unnecessary.  For a family $\mathcal F\subset L^2(I;H^1)$ define
\[
\omega_{\lambda}(\mathcal F;\delta)
:=
\sup_{u\in\mathcal F}
\sup_{\substack{K\subset I\\0<|K|\le\delta}}
|K|^{-\lambda}\int_K\|u(t)\|_{H^1}^2\,dt.
\]
We call $\mathcal F$ uniformly Morrey-absolutely continuous in $H^1$ if
\[
\omega_{\lambda}(\mathcal F;\delta)\longrightarrow0
\qquad(\delta\downarrow0).
\]
This is the uniform version of the ``little Morrey'' condition and excludes the time-bubble arrays responsible for the endpoint obstruction in the full Morrey class.

\begin{proposition}\label{prop:critical-time-cocompact}
Let $2<q<2^*$ and let $(u_n)$ be bounded in $\mathcal E(I)$ and in
$\mathcal M^{2,\lambda}(I;H^1)$.  Assume
\[
T_{y_n}u_n\rightharpoonup0
\quad\text{weakly in }\mathcal E(I)
\]
for every sequence $(y_n)\subset\mathbb R^d$, and assume that $(u_n)$ is uniformly Morrey-absolutely continuous in $H^1$.  Then
\[
\|u_n\|_{\mathcal M^{2,\lambda}(I;L^q)}\longrightarrow0.
\]
\end{proposition}

\begin{proof}
Proposition~\ref{prop:evolution-cocompact-L2} gives
\[
\|u_n\|_{L^2(I;L^q)}\to0.
\]
For intervals $K$ with $|K|\le\delta$, Sobolev embedding gives
\[
|K|^{-\lambda}\int_K\|u_n(t)\|_{L^q}^2\,dt
\le C\omega_{\lambda}(\{u_n:n\ge1\};\delta).
\]
For $|K|>\delta$,
\[
|K|^{-\lambda}\int_K\|u_n(t)\|_{L^q}^2\,dt
\le\delta^{-\lambda}\|u_n\|_{L^2(I;L^q)}^2.
\]
First let $\delta\downarrow0$ and then $n\to\infty$.
\end{proof}

\begin{corollary}\label{cor:little-morrey-profile-endpoint}
In Theorem~\ref{thm:morrey-translation-profiles}, assume in addition that the original sequence $(u_n)$ is uniformly Morrey-absolutely continuous in $H^1$.  Then for every $2<q<2^*$,
\[
\boxed{
\lim_{J\to\infty}\limsup_{n\to\infty}
\|r_n^J\|_{\mathcal M^{2,\lambda}(I;L^q)}=0.
}
\]
Thus uniform little-Morrey control removes the loss $\mu<\lambda$ from the profile remainder; the only remaining critical obstruction is the spatial endpoint $q=2^*$.
\end{corollary}

\begin{proof}
Each profile inherits the same small-interval estimate by weak lower semicontinuity.  Fix a finite number $J$ of profiles and write
\[
u_n=S_n^J+r_n^J.
\]
The spatial divergence of the profile parameters and the weak orthogonality of the residual imply
\[
\int_I\bigl|\langle S_n^J(t),r_n^J(t)\rangle_{H^1}\bigr|\,dt\to0.
\]
For fixed $J$, both $S_n^J$ and $r_n^J$ are uniformly Morrey-absolutely continuous in $H^1$: this follows from the corresponding property of the original sequence and of the finitely many profiles, together with the triangle inequality.  Hence the scalar cross term is uniformly little-Morrey in $\mathcal M^{1,\lambda}(I)$.  Splitting intervals into $|K|\le\delta$ and $|K|>\delta$, exactly as in Proposition~\ref{prop:critical-time-cocompact}, upgrades its $L^1$ convergence to
\[
\sup_{K\subset I}|K|^{-\lambda}
\int_K\bigl|\langle S_n^J,r_n^J\rangle_{H^1}\bigr|\,dt\to0.
\]
Using
\[
\|r_n^J\|_{H^1}^2
=\|u_n\|_{H^1}^2-\|S_n^J\|_{H^1}^2
-2\operatorname{Re}\langle S_n^J,r_n^J\rangle_{H^1},
\]
we therefore obtain, for every $\delta>0$,
\[
\limsup_{n\to\infty}
\sup_{\substack{K\subset I\\|K|\le\delta}}
|K|^{-\lambda}\int_K\|r_n^J(t)\|_{H^1}^2\,dt
\le
\omega_{\lambda}(\{u_n:n\ge1\};\delta),
\]
and also a uniform bound of the remainders in $\mathcal M^{2,\lambda}(I;H^1)$, with constants independent of $J$ after the outer $\limsup_n$.

If the asserted remainder convergence failed, choose $J_m\to\infty$.  For each $m$, use the fixed-$J_m$ convergence of the cross term above and the outer $\limsup_n$ in the target norm to choose $n_m$ so large that
\[
\|r_{n_m}^{J_m}\|_{\mathcal M^{2,\lambda}(I;L^q)}\ge\varepsilon
\]
for one fixed $\varepsilon>0$, while also
\[
\sup_{K\subset I}|K|^{-\lambda}
\int_K\bigl|\langle S_{n_m}^{J_m}(t),r_{n_m}^{J_m}(t)\rangle_{H^1}\bigr|\,dt
\le m^{-1}.
\]
After increasing $n_m$ once more if necessary, we also impose
\[
\left|\left\langle T_{y_{n_m}^j}U^j,T_{y_{n_m}^k}U^k\right\rangle_{\mathcal E(I)}\right|
\le \frac{2^{-m}}{J_m^2}
\qquad(1\le j<k\le J_m),
\]
which is possible because the translated profiles are pairwise asymptotically orthogonal.
The identity above then implies
\[
\sup_{\substack{K\subset I\\|K|\le\delta}}
|K|^{-\lambda}
\int_K\|r_{n_m}^{J_m}(t)\|_{H^1}^2\,dt
\le
\omega_{\lambda}(\{u_n:n\ge1\};\delta)+2m^{-1}.
\]
Hence the diagonal remainder sequence is uniformly Morrey-absolutely continuous in $H^1$.  The quantitative diagonal argument in the proof of Theorem~\ref{thm:morrey-translation-profiles}, using the square-summable profile tail and $\eta(r^K)\to0$, also shows that it is translation-weakly null in $\mathcal E(I)$.  Proposition~\ref{prop:critical-time-cocompact} therefore forces
\[
\|r_{n_m}^{J_m}\|_{\mathcal M^{2,\lambda}(I;L^q)}\longrightarrow0,
\]
contradicting the fixed lower bound.
\end{proof}

The profile theorem yields a profile-count concentration--compactness trichotomy in the subcritical Morrey topology.  Unlike the critical Morrey norm, the subcritical target cannot hide arbitrarily many order-one time concentrations.

\begin{corollary}\label{cor:morrey-profile-trichotomy}
Under the hypotheses of Theorem~\ref{thm:morrey-translation-profiles}, exactly one of the following alternatives occurs after passage to a subsequence.

\begin{enumerate}[label=\textup{(\roman*)}]
\item There is no nonzero translation profile.  Then, for every $0\le\mu<\lambda$ and $2<q<2^*$,
\[
 u_n\longrightarrow0
 \quad\text{in }\mathcal M^{2,\mu}(I;L^q).
\]
\item There is exactly one nonzero profile $U$.  Then for a suitable sequence of translations $y_n$,
\[
T_{-y_n}u_n\longrightarrow U
\quad\text{in }\mathcal M^{2,\mu}(I;L^q)
\]
for every $\mu<\lambda$ and $2<q<2^*$.
\item There are at least two nonzero profiles.  Their translation parameters diverge pairwise, their Hilbert evolution energies decouple, and the sequence splits into two or more asymptotically separated concentration components, up to a remainder vanishing in every $\mathcal M^{2,\mu}(I;L^q)$ with $\mu<\lambda$ and $q<2^*$.
\end{enumerate}
\end{corollary}

\begin{proof}
The alternatives are determined by the number of nonzero profiles in Theorem~\ref{thm:morrey-translation-profiles}.  If there are none, the whole sequence is the profile-free remainder and Theorem~\ref{thm:morrey-cocompact} gives \textup{(i)}.  If there is exactly one, translate that profile to the origin; the remaining sequence has no nonzero translation profile, so the same cocompactness theorem gives \textup{(ii)}.  If there are at least two, pairwise divergence of the translation parameters and the Hilbert energy decomposition give \textup{(iii)}.
\end{proof}

\begin{remark}\label{rem:critical-profile-boundary}
The losses $\mu<\lambda$ and $q<2^*$ are structural, not artifacts of the proof.  At $q=2^*$ spatial dilations re-enter the defect group.  At $\mu=\lambda$ the endpoint supremum over time intervals permits uniformly bounded arrays of arbitrarily many time-dislocated order-one bubbles.  Thus a universal critical remainder theorem in
$\mathcal M^{2,\lambda}_tL_x^{2^*}$ cannot hold for arbitrary bounded sequences.  The next section shows that parabolic dynamics supply exactly the additional structure needed to recover a fully critical theorem on a natural Hilbert trace subclass.
\end{remark}

\section{Critical parabolic profiles for heat and Stokes flows}\label{sec:critical-heat-profiles}

The preceding profile theorem deliberately stops below the spatial critical exponent for arbitrary evolution sequences.  At the critical endpoint, independent spatial concentration, temporal concentration, and temporal modulation are all possible in the unconstrained class.  A parabolic equation removes these independent defects.  In this section we show that, for the heat equation and hence for the whole-space Stokes system, the critical Morrey profile problem reduces exactly to the classical Hilbert profile decomposition of the initial trace.

Throughout this section let $d\ge3$, let $0<\lambda<1$, and define
\[
p_\lambda:=\frac{2d}{d-2\lambda},
\qquad
2^*:=\frac{2d}{d-2}.
\]
The exponent $p_\lambda$ is the critical Sobolev exponent for $\dot H^\lambda(\mathbb R^d)$.  Notice also that for the Hilbert couple
\[
\dot H^{-1}(\mathbb R^d),\qquad \dot H^1(\mathbb R^d),
\]
the exact Morrey trace parameter for $p=2$ is
\[
\theta=1-\frac{1-\lambda}{2}=\frac{1+\lambda}{2},
\]
which corresponds to spatial smoothness $-1+2\theta=\lambda$.  Thus $\dot H^\lambda$ is the natural Hilbert refinement of the weak trace space at this critical scale.

We first record two scale-invariant heat estimates.

\begin{lemma}\label{lem:heat-morrey-extension}
Let $f\in\dot H^\lambda(\mathbb R^d)$ and set
\[
u(t)=e^{t\Delta}f,\qquad t>0.
\]
Then
\[
\|u\|_{\mathcal M^{2,\lambda}((0,\infty);\dot H^1)}
+\|u_t\|_{\mathcal M^{2,\lambda}((0,\infty);\dot H^{-1})}
\le C_{d,\lambda}\|f\|_{\dot H^\lambda}.
\]
Moreover the two integrands agree pointwise in time:
\[
\|u_t(t)\|_{\dot H^{-1}}=\|u(t)\|_{\dot H^1}.
\]
\end{lemma}

\begin{proof}
By Plancherel,
\[
\|u_t(t)\|_{\dot H^{-1}}^2
=\int_{\mathbb R^d}|\xi|^{-2}|\xi|^4e^{-2t|\xi|^2}|\widehat f(\xi)|^2\,d\xi
=\|u(t)\|_{\dot H^1}^2.
\]
Let $J=(a,b)\subset(0,\infty)$ and write $h=b-a$.  Then
\begin{align*}
\int_a^b\|u(t)\|_{\dot H^1}^2\,dt
&=\frac12\int_{\mathbb R^d}
|\widehat f(\xi)|^2
\bigl(e^{-2a|\xi|^2}-e^{-2b|\xi|^2}\bigr)\,d\xi\\
&\le C\int_{\mathbb R^d}|\widehat f(\xi)|^2
\min\{1,h|\xi|^2\}\,d\xi.
\end{align*}
Since $0<\lambda<1$,
\[
\min\{1,s\}\le s^\lambda,
\qquad s\ge0.
\]
Hence
\[
\int_a^b\|u(t)\|_{\dot H^1}^2\,dt
\le C h^\lambda\|f\|_{\dot H^\lambda}^2.
\]
Taking $h^{-\lambda/2}$ times the square root and then the supremum over $J$ proves both Morrey estimates.
\end{proof}

\begin{lemma}\label{lem:heat-critical-smoothing}
For every $f\in L^{p_\lambda}(\mathbb R^d)$,
\[
\boxed{
\|e^{t\Delta}f\|_{\mathcal M^{2,\lambda}((0,\infty);L^{2^*})}
\le C_{d,\lambda}\|f\|_{L^{p_\lambda}}.
}
\]
\end{lemma}

\begin{proof}
The heat kernel estimate gives
\[
\|e^{t\Delta}f\|_{L^{2^*}}
\le C t^{-\frac d2(1/p_\lambda-1/2^*)}\|f\|_{L^{p_\lambda}}
=Ct^{-(1-\lambda)/2}\|f\|_{L^{p_\lambda}}.
\]
Let $J=(a,b)$ and $h=b-a$.  Then
\[
\int_a^b t^{-(1-\lambda)}\,dt
=\frac{b^\lambda-a^\lambda}{\lambda}
\le \frac{h^\lambda}{\lambda},
\]
because $s\mapsto s^\lambda$ is concave and subadditive on $[0,\infty)$.  Therefore
\[
h^{-\lambda}\int_J\|e^{t\Delta}f\|_{L^{2^*}}^2\,dt
\le C_{d,\lambda}\|f\|_{L^{p_\lambda}}^2.
\]
Taking the supremum over $J$ proves the claim.
\end{proof}

The preceding lemma is the bridge from the classical Sobolev profile decomposition to the fully critical Morrey evolution target.  We use the profile theorem of G\'erard~\cite{Gerard1998}; see also the fractional Sobolev treatment of Palatucci--Pisante~\cite{PalatucciPisante2014}.

For $h>0$ and $y\in\mathbb R^d$ define the critical spatial dislocation
\[
(g_{h,y}\phi)(x)
:=h^{-d/2+\lambda}\phi\!\left(\frac{x-y}{h}\right).
\]
This acts isometrically on $\dot H^\lambda$ and on $L^{p_\lambda}$.  Its parabolic lift is
\[
(\mathcal G_{h,y}U)(t,x)
:=h^{-d/2+\lambda}
U\!\left(\frac{t}{h^2},\frac{x-y}{h}\right).
\]
The heat equation intertwines these two actions:
\[
e^{t\Delta}(g_{h,y}\phi)
=\mathcal G_{h,y}(e^{t\Delta}\phi).
\]

\begin{theorem}\label{thm:critical-heat-profile}
Let $(f_n)$ be bounded in $\dot H^\lambda(\mathbb R^d)$ and set
\[
u_n(t)=e^{t\Delta}f_n.
\]
After passing to a subsequence, there exist profiles
\[
\phi^j\in\dot H^\lambda(\mathbb R^d),\qquad j\ge1,
\]
scales $h_n^j>0$, centers $y_n^j\in\mathbb R^d$, and remainders $r_n^J\in\dot H^\lambda$ such that, for every finite $J$,
\[
f_n
=\sum_{j=1}^J g_{h_n^j,y_n^j}\phi^j+r_n^J.
\]
The parameters are asymptotically orthogonal: for $j\ne k$,
\[
\frac{h_n^j}{h_n^k}
+\frac{h_n^k}{h_n^j}
+\frac{|y_n^j-y_n^k|}{h_n^j}
\longrightarrow\infty.
\]
For every fixed $J$,
\[
\|f_n\|_{\dot H^\lambda}^2
=\sum_{j=1}^J\|\phi^j\|_{\dot H^\lambda}^2
+\|r_n^J\|_{\dot H^\lambda}^2+o_n(1),
\]
and
\[
\lim_{J\to\infty}\limsup_{n\to\infty}
\|r_n^J\|_{L^{p_\lambda}}=0.
\]

If
\[
U^j(t):=e^{t\Delta}\phi^j,
\qquad
R_n^J(t):=e^{t\Delta}r_n^J,
\]
then the heat solutions admit the balanced parabolic decomposition
\[
\boxed{
u_n
=\sum_{j=1}^J\mathcal G_{h_n^j,y_n^j}U^j+R_n^J,}
\]
and the remainder vanishes in the fully critical Morrey target:
\[
\boxed{
\lim_{J\to\infty}\limsup_{n\to\infty}
\|R_n^J\|_{\mathcal M^{2,\lambda}((0,\infty);L^{2^*})}=0.
}
\]
Each profile $U^j$ belongs to the critical Morrey evolution class of Lemma~\ref{lem:heat-morrey-extension}.
\end{theorem}

\begin{proof}
Apply G\'erard's homogeneous Sobolev profile decomposition to the bounded sequence $(f_n)$ in $\dot H^\lambda$.  Since $0<\lambda<d/2$, the critical target is exactly $L^{p_\lambda}$.  This yields the spatial decomposition, parameter orthogonality, the Hilbert Pythagorean expansion, and the $L^{p_\lambda}$-small remainder.

Apply the heat semigroup to the decomposition.  Linearity and the scaling identity give
\[
u_n
=\sum_{j=1}^J\mathcal G_{h_n^j,y_n^j}U^j+R_n^J.
\]
Lemma~\ref{lem:heat-morrey-extension} places every profile in the critical Morrey evolution class.  Finally Lemma~\ref{lem:heat-critical-smoothing} gives
\[
\|R_n^J\|_{\mathcal M^{2,\lambda}_tL_x^{2^*}}
\le C_{d,\lambda}\|r_n^J\|_{L^{p_\lambda}}.
\]
Taking first $\limsup_{n\to\infty}$ and then $J\to\infty$ proves the critical remainder statement.
\end{proof}

\begin{corollary}\label{cor:critical-heat-trichotomy}
For bounded heat-flow sequences generated by bounded subsets of $\dot H^\lambda$, the profile theorem yields the following critical concentration alternatives.
\begin{enumerate}[label=\textup{(\roman*)}]
\item If all profiles vanish, then
\[
\|u_n\|_{\mathcal M_t^{2,\lambda}L_x^{2^*}}\to0.
\]
\item If there is exactly one nonzero profile, then after undoing one spatial translation and dilation at the initial time, the corresponding heat flows converge to that profile modulo a remainder vanishing in the critical Morrey target.
\item If there are at least two nonzero profiles, the initial $\dot H^\lambda$ energy decouples among pairwise orthogonal spatial translation--dilation profiles, while the heat-flow remainder is still negligible in the critical Morrey target.
\end{enumerate}
\end{corollary}

\begin{proof}
This is an immediate reformulation of Theorem~\ref{thm:critical-heat-profile} according to the number of nonzero profiles.
\end{proof}

The same result applies to the Stokes semigroup on the whole space.  Let
\[
\dot H^\lambda_\sigma(\mathbb R^d)
:=\{f\in\dot H^\lambda(\mathbb R^d;\mathbb R^d):\operatorname{div}f=0\}.
\]

\begin{corollary}\label{cor:critical-stokes-profile}
Let $(f_n)$ be bounded in $\dot H^\lambda_\sigma(\mathbb R^d)$ and let
\[
u_n(t)=e^{-tA}f_n
\]
be the whole-space Stokes flows.  Then Theorem~\ref{thm:critical-heat-profile} holds with divergence-free vector-valued profiles $\phi^j\in\dot H^\lambda_\sigma$ and with the same critical remainder conclusion
\[
\lim_{J\to\infty}\limsup_{n\to\infty}
\|R_n^J\|_{\mathcal M^{2,\lambda}((0,\infty);L^{2^*})}=0.
\]
\end{corollary}

\begin{proof}
On divergence-free vector fields in $\mathbb R^d$, the Stokes semigroup is the heat semigroup.  Spatial translations and dilations preserve the divergence-free condition.  The divergence-free subspace is weakly closed in $\dot H^\lambda$, so every profile obtained as a weak limit of rescaled and translated divergence-free data is divergence free.  The remainder is divergence free as well.  Theorem~\ref{thm:critical-heat-profile} therefore applies componentwise.
\end{proof}

\begin{remark}\label{rem:why-hilbert-trace-closes-critical}
The critical theorem above explains why the unrestricted Morrey evolution class and the heat/Stokes solution class behave differently.  The exact Morrey trace space for the Hilbert couple is a weak Besov-type space with fine index $\infty$, whose endpoint structure permits nonsummable arrays of critical defects.  Restricting the trace to the scale-invariant Hilbert space $\dot H^\lambda$ restores square summability of profile energies.  The parabolic equation then removes the independent temporal scale and modulation parameters, and Lemma~\ref{lem:heat-critical-smoothing} converts the classical $L^{p_\lambda}$ profile remainder into the exact critical time-Morrey remainder.  Thus the balanced parabolic profile theorem is valid on this natural Hilbert trace subclass even though it fails for arbitrary bounded Morrey evolutions.
\end{remark}

\section{A Morrey refinement of nonlinear Navier--Stokes profiles}\label{sec:ns-morrey-profiles}
The critical heat/Stokes theorem has a nonlinear consequence in three dimensions which can be obtained without rebuilding the entire perturbation argument.  Gallagher's nonlinear profile theorem already supplies the orthogonal nonlinear profiles together with two remainders of different types.  Both remainders are automatically small in a continuum of critical time--Morrey spaces.

For $2\le p<4$ put
\[
\lambda_p:=1-\frac p4,
\qquad
Y_p(T):=\mathcal M^{p,\lambda_p}(0,T;L^6(\mathbb R^3)).
\]
These spaces are invariant under the Navier--Stokes scaling
\[
(\mathcal G_{h,y}U)(t,x)
=h^{-1}U(t/h^2,(x-y)/h).
\]
The identity
\[
\frac{1-\lambda_p}{p}=\frac14
\]
is the critical balance.

For the global alternative in the classical nonlinear profile theorem we record one auxiliary fact.  It is specific to actual Navier--Stokes trajectories and should not be confused with a density statement for arbitrary paths in the ambient energy space.

\begin{lemma}\label{lem:global-profile-decay}
Let $V$ be a global mild three-dimensional Navier--Stokes solution satisfying
\[
V\in E_\infty
:=L^\infty(0,\infty;\dot H^{1/2})\cap L^2(0,\infty;\dot H^{3/2}).
\]
Then
\[
\|V(t)\|_{\dot H^{1/2}}\longrightarrow0
\qquad(t\to\infty).
\]
Consequently $V$ can be approximated in $E_\infty$ by trajectories compactly supported in time.  On each finite time window one may then use smooth rapidly decaying spatial approximants; for the Morrey bilinear estimate below, the weaker space $L_t^4L_x^6$ admits smooth compactly supported spacetime approximation.
\end{lemma}

\begin{proof}
Interpolation gives $V\in L^4(0,\infty;\dot H^1)$.  Restart the mild equation at a time $T>0$:
\[
V(t)=e^{(t-T)\Delta}V(T)
-\int_T^t e^{(t-s)\Delta}\mathbb P\nabla\!\cdot(V\otimes V)(s)\,ds.
\]
The forced Stokes estimate together with
\[
\|V\otimes V\|_{\dot H^{1/2}}
\lesssim \|V\|_{\dot H^1}^2
\]
gives
\[
\sup_{t\ge T}
\left\|
\int_T^t e^{(t-s)\Delta}\mathbb P\nabla\!\cdot(V\otimes V)(s)\,ds
\right\|_{\dot H^{1/2}}
\le C\|V\|_{L^4(T,\infty;\dot H^1)}^2,
\]
and the right-hand side tends to zero as $T\to\infty$.  For fixed $T$, the heat orbit $e^{(t-T)\Delta}V(T)$ tends to zero in $\dot H^{1/2}$ by dominated convergence on the Fourier side.  Thus $\|V(t)\|_{\dot H^{1/2}}\to0$.

Let $\chi_T$ be a smooth cutoff equal to one on $[0,T]$ and zero on $[T+1,\infty)$.  The $L^\infty_t\dot H^{1/2}$ tail decay just proved and the $L^2_t\dot H^{3/2}$ tail integrability imply
\[
\|V-\chi_TV\|_{E_\infty}\to0.
\]
This proves the time-compact approximation.  The remaining approximation statements follow from standard density in the indicated finite-time Sobolev and $L_t^4L_x^6$ spaces.
\end{proof}

\begin{lemma}\label{lem:energy-to-morrey-kato}
Let
\[
E_T=L^\infty(0,T;\dot H^{1/2}(\mathbb R^3))
\cap L^2(0,T;\dot H^{3/2}(\mathbb R^3)).
\]
Then for every $2\le p\le4$,
\[
\|v\|_{Y_p(T)}
\le C
\|v\|_{L^\infty_t\dot H^{1/2}_x}^{1/2}
\|v\|_{L^2_t\dot H^{3/2}_x}^{1/2}.
\]
The constant is independent of $T$.
\end{lemma}

\begin{proof}
Interpolation in the homogeneous Hilbert scale gives
\[
\|v(t)\|_{\dot H^1}
\le
\|v(t)\|_{\dot H^{1/2}}^{1/2}
\|v(t)\|_{\dot H^{3/2}}^{1/2}.
\]
Hence
\[
\|v\|_{L^4_t\dot H^1_x}
\le
\|v\|_{L^\infty_t\dot H^{1/2}_x}^{1/2}
\|v\|_{L^2_t\dot H^{3/2}_x}^{1/2}.
\]
Using $\dot H^1(\mathbb R^3)\hookrightarrow L^6$ it remains to observe that
\[
L^4(0,T;X)\hookrightarrow\mathcal M^{p,\lambda_p}(0,T;X),
\qquad 2\le p\le4.
\]
Indeed, on every interval $I$,
\[
|I|^{-\lambda_p/p}\|f\|_{L^p(I;X)}
\le
|I|^{-\lambda_p/p+1/p-1/4}\|f\|_{L^4(I;X)}
=\|f\|_{L^4(I;X)}.
\]
\end{proof}

\begin{lemma}\label{lem:heat-l3-morrey-kato}
For every $2\le p<4$ and every $f\in L^3(\mathbb R^3)$,
\[
\|e^{t\Delta}f\|_{Y_p(\infty)}\le C_p\|f\|_{L^3}.
\]
\end{lemma}

\begin{proof}
The heat bound
\[
\|e^{t\Delta}f\|_6\lesssim t^{-1/4}\|f\|_3
\]
and the identity $p/4=1-\lambda_p<1$ give, for every interval $(a,b)$ of length $h$,
\[
\int_a^b t^{-p/4}\,dt
\le C_p h^{\lambda_p}.
\]
This is exactly the claimed Morrey estimate.
\end{proof}

\begin{theorem}\label{thm:ns-morrey-profile}
Let $(\phi_n)$ be bounded in the divergence-free space $\dot H^{1/2}_\sigma(\mathbb R^3)$.  After extraction, let
\[
\phi_n
=\sum_{j=0}^{\ell}\frac1{h_n^j}
\phi^j\!\left(\frac{x-x_n^j}{h_n^j}\right)
+\psi_n^\ell
\]
be the G\'erard decomposition used in Gallagher's nonlinear profile theorem, and let $V^j=NS(\phi^j)$.  Let $\tau_n$ be the common time supplied by that theorem.  Then for every $2\le p<4$,
\[
NS(\phi_n)(t,x)
=\sum_{j=0}^{\ell}\frac1{h_n^j}
V^j\!\left(
\frac{t}{(h_n^j)^2},
\frac{x-x_n^j}{h_n^j}
\right)
+R_n^\ell(t,x)
\]
on $(0,\tau_n)$, with
\[
\lim_{\ell\to\infty}\limsup_{n\to\infty}
\|R_n^\ell\|_{Y_p(\tau_n)}=0.
\]
Whenever Gallagher's global alternative applies, the same conclusion holds with $\tau_n=\infty$; Lemma~\ref{lem:global-profile-decay} supplies the global-time approximation property needed for the actual nonlinear profiles.
\end{theorem}

\begin{proof}
Gallagher's theorem gives
\[
R_n^\ell=w_n^\ell+r_n^\ell,
\qquad
w_n^\ell=e^{t\Delta}\psi_n^\ell,
\]
where the initial G\'erard remainder satisfies
\[
\lim_{\ell\to\infty}\limsup_n
\|\psi_n^\ell\|_{L^3}=0,
\]
and the nonlinear remainder satisfies
\[
\lim_{\ell\to\infty}\limsup_n
\|r_n^\ell\|_{E_{\tau_n}}=0.
\]
Lemma~\ref{lem:heat-l3-morrey-kato} makes $w_n^\ell$ small in $Y_p$ uniformly in $\tau_n$, while Lemma~\ref{lem:energy-to-morrey-kato} makes $r_n^\ell$ small in the same norm.  The conclusion follows by the triangle inequality.  Notice that the Morrey topology itself is not imported from Gallagher's interaction estimates: Theorem~\ref{thm:morrey-cross-interaction} proves the corresponding orthogonality directly in the natural Morrey forcing space.
\end{proof}

\begin{remark}\label{rem:ns-morrey-endpoint}
The endpoint $p=2$, $\lambda_2=1/2$, is included in Theorem~\ref{thm:ns-morrey-profile}.  Thus $\mathcal M_t^{2,1/2}L_x^6$ is a valid critical profile-remainder topology even though the direct quadratic contraction argument in this space encounters the endpoint $\mathcal M^{1,1/2}$ fractional-integral problem.  For $2<p<4$, the spaces $Y_p$ are also strong perturbative Navier--Stokes spaces via the Adams Morrey fractional-integral estimate.
\end{remark}

The Morrey refinement is compatible with the profile geometry itself.  We make the interaction statement explicit.

\begin{theorem}\label{thm:morrey-cross-interaction}
Let $U\in E_{T_1}$ and $V\in E_{T_2}$, where
\[
E_T=L^\infty(0,T;\dot H^{1/2})\cap L^2(0,T;\dot H^{3/2}).
\]
Let $(h_n,x_n)$ and $(k_n,y_n)$ be orthogonal G\'erard parameters and set
\[
U_n(t,x)=h_n^{-1}U(t/h_n^2,(x-x_n)/h_n),
\qquad
V_n(t,x)=k_n^{-1}V(t/k_n^2,(x-y_n)/k_n).
\]
On every common time interval on which the two scaled profiles are defined, for every $2\le p\le4$,
\[
\boxed{
\|U_nV_n\|_{\mathcal M_t^{p/2,\lambda_p}L_x^3}\longrightarrow0.
}
\]
\end{theorem}

\begin{proof}
By Lemma~\ref{lem:energy-to-morrey-kato}, $E_T\hookrightarrow L_t^4L_x^6\hookrightarrow Y_p$, and smooth compactly supported functions are dense in $L_t^4L_x^6$.  Since H\"older's inequality on every time interval gives
\[
|J|^{-2\lambda_p/p}\|ab\|_{L_t^{p/2}(J;L_x^3)}
\le
\bigl(|J|^{-\lambda_p/p}\|a\|_{L_t^p(J;L_x^6)}\bigr)
\bigl(|J|^{-\lambda_p/p}\|b\|_{L_t^p(J;L_x^6)}\bigr),
\]
we have
\[
\|ab\|_{\mathcal M_t^{p/2,\lambda_p}L_x^3}
\le \|a\|_{Y_p}\|b\|_{Y_p}.
\]
This remains valid at $p=2$, where the forcing space is the Banach Morrey space $\mathcal M_t^{1,1/2}L_x^3$; no fractional-integral endpoint estimate is being used in the present theorem.  It is therefore enough to prove the assertion for smooth compactly supported $U$ and $V$.

Suppose first, after interchanging the profiles if necessary, that
\[
\varepsilon_n:=h_n/k_n\to0.
\]
The forcing norm is invariant under
\[
F(t,x)\mapsto h_n^2F(h_n^2t,x_n+h_nx).
\]
In these coordinates the product becomes
\[
U(t,x)\,\varepsilon_n
V\!\left(\varepsilon_n^2t,
\frac{x_n-y_n}{k_n}+\varepsilon_n x\right).
\]
On the fixed compact support of $U$ the second factor is $O(\varepsilon_n)$ uniformly, hence the product tends to zero in the Morrey forcing norm.  If the two scales are comparable, the G\'erard orthogonality convention allows the profiles, after absorbing a fixed limiting scale ratio into one profile, to be represented with the same scale.  Orthogonality then forces the normalized spatial cores to diverge.  After rescaling by that common scale, the compact spatial supports are disjoint for all large $n$.  Approximation in $L_t^4L_x^6$ completes the proof, including the endpoint $p=2$.
\end{proof}

For $2<p<4$ the interaction can be propagated through the Duhamel operator in the same Morrey scale.

\begin{lemma}\label{lem:adams-duhamel}
Let $2<p<4$, $\lambda_p=1-p/4$, and
\[
Z_p(T):=\mathcal M^{p/2,\lambda_p}(0,T;L^3(\mathbb R^3)).
\]
Then
\[
\left\|\int_0^t e^{(t-s)\Delta}\mathbb P\nabla\!\cdot F(s)\,ds\right\|_{Y_p(T)}
\le C_p\|F\|_{Z_p(T)},
\]
with a constant independent of $T$.  Consequently
\[
\|B(u,v)\|_{Y_p(T)}\le C_p\|u\|_{Y_p(T)}\|v\|_{Y_p(T)}.
\]
\end{lemma}

\begin{proof}
The heat kernel gives
\[
\|e^{(t-s)\Delta}\mathbb P\nabla\!\cdot F(s)\|_6
\le C(t-s)^{-3/4}\|F(s)\|_3.
\]
Thus the time variable is controlled by the one-dimensional Riesz potential $I_{1/4}$.  In Adams' notation~\cite{Adams1975}, our Morrey parameter $\lambda_p$ corresponds to
\[
\Lambda=1-\lambda_p=\frac p4.
\]
With input exponent $p/2$ and $\alpha=1/4$, Adams' strong theorem gives
\[
I_{1/4}:\mathcal M^{p/2,\lambda_p}\longrightarrow
\mathcal M^{p,\lambda_p},
\qquad
\frac1p=\frac2p-\frac{1/4}{p/4}.
\]
The strict condition $p/2>1$ is exactly $p>2$.  Extending $F$ by zero outside $(0,T)$ and dominating the one-sided convolution by the full Riesz potential proves the first estimate.  The second follows from
\[
\|u\otimes v\|_{Z_p(T)}\le\|u\|_{Y_p(T)}\|v\|_{Y_p(T)}.
\]
\end{proof}

\begin{corollary}\label{cor:morrey-cross-duhamel}
Under the assumptions of Theorem~\ref{thm:morrey-cross-interaction}, for $2<p<4$,
\[
\|B(U_n,V_n)+B(V_n,U_n)\|_{Y_p}\longrightarrow0.
\]
\end{corollary}

\begin{proof}
Combine Theorem~\ref{thm:morrey-cross-interaction} with Lemma~\ref{lem:adams-duhamel}.
\end{proof}

\section{Further consequences and scope}

The results above separate three compactness mechanisms that are often grouped under the name Aubin--Lions.

First, the classical spatial compact embedding
\[
E_0\hookc E\hookrightarrow E_1
\]
controls strong convergence in the \emph{same Morrey norm}; no trace-space compactness is needed.  Second, compactness in $C([0,T];E)$ is governed exactly by the compact embedding of the trace space $X_\theta$ into $E$.  Third, on unbounded domains local compactness must be supplemented by tightness in order to exclude translation to infinity.

The Hilbert profile theorem above identifies the translation defect below the critical spatial exponent and, under uniform little-Morrey control, at the original time-Morrey exponent as well.  At the doubly critical target $\mathcal M_t^{2,\lambda}L_x^{2^*}$, arbitrary bounded evolutions admit additional spatial, temporal, and modulation defects.  The heat/Stokes theorem shows that parabolic dynamics remove these independent defects on the natural Hilbert trace subclass, and Theorem~\ref{thm:ns-morrey-profile} shows that the corresponding nonlinear Navier--Stokes profiles inherit an entire critical Morrey--Kato remainder scale.  A natural next problem is to combine these Morrey profile topologies with local $\varepsilon$-regularity in order to localize critical packets for arbitrary singular sequences; that question is not needed for the compactness theory proved here.

For nonlinear parabolic equations the same-Morrey compactness theorem is directly suited to approximation schemes.  For example, on a bounded Lipschitz domain $\Omega$, the choice
\[
E_0=W^{1,r}(\Omega),
\qquad E=L^r(\Omega),
\qquad E_1=W^{-1,r}(\Omega)
\]
gives
\[
\mathbb W_M^{p,\lambda}
\hookc
\mathcal M^{p,\lambda}(0,T;L^r(\Omega)).
\]
This is the strong convergence topology naturally needed to pass to many nonlinear terms while preserving the time-Morrey scale.

\section*{Funding}
This research did not receive any specific grant from funding agencies in the public, commercial, or not-for-profit sectors.

\section*{Declaration of competing interest}
The authors declare no competing interests.

\section*{Data availability}
No datasets were generated or analysed during the current study.

\section*{Declaration of generative AI and AI-assisted technologies in the manuscript preparation process}
During the preparation of this work, the authors used OpenAI ChatGPT for editorial assistance, literature organization, and checking the exposition and internal consistency of mathematical arguments. After using this tool, the authors reviewed and edited the content as needed and take full responsibility for the content of the publication.

\end{document}